\documentclass[11pt,reqno]{article}
\usepackage{amsthm,amsmath, amssymb, amsopn, amsfonts,enumitem}
\usepackage{cases}
\usepackage[margin=1in]{geometry}
\usepackage{subcaption}
\usepackage[colorlinks,citecolor=blue,urlcolor=blue]{hyperref}
\usepackage{graphicx}
\usepackage{tikz}
\usetikzlibrary{decorations.fractals}
\usetikzlibrary{patterns,arrows,shapes,decorations.pathreplacing}
\usepackage{pgfplots}
\usepackage{stmaryrd} 
\usepackage{soul,xcolor}
\setstcolor{red}
\usepackage{xcolor}

\numberwithin{equation}{section}
\theoremstyle{plain}
\newtheorem{Definition}{Definition}[section]
\newtheorem{Remark}{Remark}[section]
\newtheorem{Theorem}{Theorem}[section]
\newtheorem{Lemma}{Lemma}[section]
\newtheorem{Proposition}{Proposition}[section]

\newtheorem{Assumption}{Assumption}[section]

\newcommand{\be}{\begin{equation}}
\newcommand{\ee}{\end{equation}}
\newcommand{\bee}{\begin{equation*}}
\newcommand{\eee}{\end{equation*}}
\newcommand{\bi}{\begin{itemize}}
\newcommand{\ei}{\end{itemize}}
\def \E{\mathbb{E}}
\def \F{\mathbb{F}}
\def \N{\mathbb{N}}
\def \P{\mathbb{P}}

\def \R{\mathbb{R}}
\def \X{\mathbb{X}}

\def \Bc{{\mathcal B}}
\def \Ec{{\mathcal E}}
\def \Fc{{\mathcal F}}

\def \eps{\varepsilon}

\makeatletter
\newcommand{\setword}[2]{%
	\phantomsection
	#1\def\@currentlabel{\unexpanded{#1}}\label{#2}%
}
\makeatother

\title{Stability of Time-inconsistent Stopping for One-dimensional Diffusion - A Longer Version}

\author{Erhan Bayraktar\thanks{
		Department of Mathematics, University of Michigan, Ann Arbor, email: \texttt{erhan@umich.edu}. E. Bayraktar is partially supported by the National Science Foundation under grant DMS2106556 and by the Susan M. Smith chair.}	
	\and Zhenhua Wang\thanks{
		Department of Mathematics, University of Michigan, Ann Arbor, email: \texttt{zhenhuaw@umich.edu}. }
	\and Zhou Zhou\thanks{School of Mathematics and Statistics, University of Sydney, Australia, email:
		\texttt{zhou.zhou@sydney.edu.au}.}
}

\begin{document}

\maketitle

\date{}

\begin{abstract}
We investigate the stability of the equilibrium-induced optimal value in a one-dimensional diffusion setting for a time-inconsistent stopping problem under non-exponential discounting. We show that the optimal value is semi-continuous with respect to the drift, volatility, and reward function. An example is provided showing that the exact continuity may fail. With equilibria extended to $\eps$-equilibria, we establish the relaxed continuity of the optimal value.
\end{abstract}	
{\bf Keywords:} Time-inconsistency, Optimal equilibrium, $\eps$-equilibria, Stability\\
{\bf MSC(2020):}
49K40, 
60G40, 
91A11, 
91A15. 

\section{Introduction}
The study of time-inconsistent stopping has attracted considerable attention recently. See \cite{MR4250561,huang2018time,MR4273542,MR4116459,MR4080735,MR3880244,MR4205889,MR4332966,MR3980261,bayraktar2022equilibria,MR3911711} and the references therein. Among them, \cite{huang2018time} provides a general framework for time-inconsistent stopping in continuous time. The notion of equilibria in \cite{huang2018time} (called mild equilibria since \cite{MR4205889}) is further investigated in e.g., \cite{MR4250561,MR4273542,MR4116459}. In particular, it is shown in \cite{MR4250561,MR4116459} that there exists an optimal mild equilibrium which pointwisely dominates any other mild equilibrium. Another concept of equilibria (called weak equilibria in \cite{MR4205889}) is proposed using a first order condition in \cite{MR3880244}. Such kind of equilibria are typically characterized by some extended HJB equation system. See. e.g., \cite{MR4080735,MR3880244,MR4332966}. In \cite{MR4205889}, a third notion of equilibria, called strong equilibria, is proposed, which better captures the economic meaning of being ``equilibria". {\color{black} A further description of mild, weak and strong equilibria is relegated to Appendix \ref{appendix}.} In \cite{MR4205889} it is shown that an optimal mild equilibrium is also weak and strong in a continuous Markov chain setting under non-exponential discounting. Recently, \cite{bayraktar2022equilibria} extends such result to the one-dimensional diffusion case. {\color{black} Let us also mention that pure strategies are studied in \cite{huang2018time, MR4332966,bayraktar2022equilibria, MR4205889, MR4250561,MR4273542,MR4116459}, while mixed-type equilibria are investigated in \cite{MR4080735, MR4276004, bodnariu2022local}.}

In this paper, in the one-dimensional diffusion infinite-horizon setting under weighted (and thus non-exponential) discounting, we consider the stability of the optimal value induced by all pure mild equilibria (denote as $V^{\mu,\sigma}(\cdot, f)$) with respect to  (w.r.t.) the drift $\mu$, volatility $\sigma$ and reward function $f$. We show that the optimal value w.r.t. $(\mu,\sigma,f)$, i.e., $(\mu,\sigma,f)\mapsto V^{\mu,\sigma}(\cdot,f)$, is upper semi-continuous. We provide an example showing that the exact continuity may fail. In order to recover the continuity, we relax the equilibrium set and consider $\eps$-mild equilibria. Thanks to this relaxation, we establish the continuity in the sense that $\lim_{\eps\searrow 0}\lim_{n\to\infty}V_\eps^{\mu^n,\sigma^n}(\cdot,f^n)=V^{\mu,\sigma}(\cdot,f)$ when $(\mu^n,\sigma^n,f^n)\to (\mu,\sigma,f)$ in certain sense, where $V_\eps^{\mu^n,\sigma^n}(\cdot,f^n)$ is the optimal value generated by all $\eps$-mild equilibria w.r.t. $(\mu^n,\sigma^n,f^n)$.

Our paper extends the results in \cite{bayraktar2022stability} to the one-dimensional diffusion case. Compared to \cite{bayraktar2022stability}, a major difference is the mathematical approach: in this paper we need to apply different methods to establish intermediate results, including a PDE approach for the uniform convergence of some stopping value functions. Another difference is related to the semi-continuity for the smallest mild equilibrium (which is an optimal one). In \cite{bayraktar2022stability} it is shown that the smallest mild equilibrium is lower semi-continuous w.r.t. the law of the underlying process and the reward function in discrete time, while in this paper we provide an example showing that such semi-continuity may fail in the {\color{black} diffusion} framework.

The literature on stability analysis for Nash games is very sparse. Let us mention the very recent works \cite{feinstein2020continuity} and \cite{feinstein2022dynamic} on this topic. In the research of time-inconsistent stopping, to the best of our knowledge, only \cite{bayraktar2022stability,MR4276004} have studied the stability before, yet the notion of stability in \cite{MR4276004} differs from that in our paper. Given the difference between this paper and \cite{bayraktar2022stability}, and limited literature in this topic, we believe our results are novel and significant.

The rest of the paper is organized as follows. Section \ref{sec:preliminaries} provides the setup. The main results are introduced in Section \ref{sec:stability}, including the semi-continuity of the optimal value function w.r.t $(\mu,\sigma,f)$, and {\color{black} the stability of the value function when relaxing the equilibrium set.} In Section \ref{sec:example}, we provide two examples, one for the strict semi-continuity of $(\mu,\sigma,f)\mapsto V^{\mu,\sigma}(\cdot,f)$, the other for the failure of the semi-continuity for the smallest mild equilibrium w.r.t $(\mu,\sigma,f)$. {\color{black}Appendix \ref{appendix} provides a brief introduction of mild, weak and strong equilibria.}

\section{Setup and Preliminaries}\label{sec:preliminaries}

Let $(\Omega, \P, (\Fc_t)_{t}, \F)$ be a filtered probability space supporting a 1-dimensional Brownian motion $W$. Let $\X\subset \R$ be an open interval and $\Bc$ be the class of Borel measurable subsets of $\X$. For $A\in \Bc$, denote by $\overline{A}$ the closure of $A$ (w.r.t. the Euclidean topology induced by $\X$). $\R_+$ (resp. $\N$) denotes the set of non-negative real numbers (resp. all positive integers), and set $\overline\N:=\N\cup\{\infty\}$. By convention $\frac{1}{\infty}=0$. For a function $g: \X\to \R$, set $\|g\|_{\infty}:=\sup_{x\in \X} |g(x)|$. We further set $Q:= (\mu,\sigma)$ for two functions $\mu,\sigma: \X\mapsto \R$ such that a 1-dimensional diffusion $X$ given by
\begin{equation}\label{e1}
	dX_t=\mu(X_t)dt+\sigma(X_t)dW_t,
\end{equation}
is supported on $\X$ for any $X_0=x\in\X$.
\begin{Definition}
	$Q=(\mu,\sigma)$ is said to be regular, if $\mu, \sigma$ are Lipschitz continuous and $|\sigma(\cdot)|>0$.
\end{Definition}

{\color{black}Throughout this paper, we always assume $Q$ is regular and such that $X$ given by \eqref{e1} is supported on $\X$.} Denote by $\E^Q_x[\cdot ]$ (resp. $\P^Q_x(\cdot)$) the expectation (resp. probability) associated with $Q$ and $X_0=x$.

Let $\delta:\R_+\mapsto[0,1]$ be a discount function that is strictly decreasing and $\lim_{t\to\infty}\delta(t)=0$. We make the following assumption on $\delta$.
\begin{Assumption}\label{assume.delta.weight}
	$\delta(\cdot)$ is a weighted discount function of the form:
	$
	\delta(t)= \int_{0}^\infty e^{-r t}F(dr),
	$
	where $F(r): [0,\infty)\to [0,1]$ is a cumulative distribution function.
\end{Assumption}

\begin{Remark}
	Most commonly used discount functions obey the weighted discounting form. See e.g., \cite{MR4124420} for a detailed discussion. Moreover, \cite[Proposition 1]{MR4124420} indicates that all weighted discount functions satisfy the following decreasing impatience property:
\be\label{eq.delta.logsub} 
\delta(t+s)\geq \delta(t)\delta(s)\quad \forall\, t,s\geq 0,
\ee
{\color{black}In addition, pure and mixed weak equilibria and the corresponding smooth fit property under weighted discounting have been investigated in \cite{MR4332966} and \cite{bodnariu2022local} respectively.}
\end{Remark}

For $A\in \Bc$, let
$
\rho_{A}:=\inf\{t>0: X_t\in A\}.
$
Given $Q=(\mu,\sigma)$, a reward function $f:\X\mapsto\R_+$, and $A\in \Bc$, define 
$$ 
J^Q(x, A, f):= \E^Q_x[\delta(\rho_A) f(X_{\rho_A})]\quad \forall x\in \X.
$$
Recall the notion of mild equilibria and optimal mild equilibria defined in \cite{bayraktar2022equilibria} as follows.
\begin{Definition}[Mild equilibria and optimal mild equilibria]\label{def.mild.optimal}
	A closed set $S\subset\X$ is said to be  a mild equilibrium (w.r.t. $f$ and $Q$), if
\be \label{e3} f(x)\leq J^Q(x,S,f)\quad \forall x\notin S.\ee
Denote by $\Ec^Q(f)$ the set of mild equilibria w.r.t. $(f,Q)$.
A mild equilibrium $S$ is  said to be optimal, if for any other mild equilibrium $R\in\Ec^Q(f)$, 
$$J^Q(x,S,f)\geq J^Q(x,R,f)\quad \forall x\in\X.$$
\end{Definition}
\begin{Remark}
	{\color{black}$|\sigma|>0$ implies that $\rho_{\{ x \}}=0, \mathbb{P}^x$-a.s. for any $x\in \mathbb{X}$. Thus, $\rho_A= \rho_{\overline{A}}$ for any $ A\in \Bc$.
	This is why we restrict equilibria to be closed. Moreover, this also indicates that $f(x)=J^Q(x,A,f)$ for any $x\in \overline{A}$. Consequently, there is no need to consider the condition
	$$f(x)\geq J^Q(x,S,f),\quad \forall x\in S$$
	as being part of the requirement for mild equilibria. We refer to Appendix \ref{appendix} for a detailed discussion.}
\end{Remark}

\begin{Remark}
	It is shown in \cite[Theorem 4.1]{MR4116459} that the smallest mild equilibrium is optimal. We rewrite this result as a lemma in the following, since it will be used later in the paper. It is proved in \cite{bayraktar2022equilibria,MR4205889} that under mild assumptions an optimal mild equilibrium is also weak and strong.
\end{Remark}

\begin{Lemma}\label{lm.optimalmild}
	Let Assumption \ref{assume.delta.weight} (or \eqref{eq.delta.logsub}) hold. Suppose $Q=(\mu,\sigma)$ is regular and $f$ is non-negative, continuous and $\|f\|_\infty<\infty$. Then 
	$$S^*(f,Q):= \cap_{S\in \Ec^Q(f)} S$$
	is an optimal mild equilibrium.
\end{Lemma}

Let $V^{Q}(x,f)$ be the optimal value generated over all mild equilibria, i.e.,
\be\label{eq.def.V}  
V^{Q}(x,f):=\sup_{S\in \Ec^Q(f)} J^Q(x,S,f)\quad \forall x\in \X.                 
\ee
Under the assumption in Lemma \ref{lm.optimalmild}, we have $V^{Q}(x,f)=J^Q(x, S^*(f,Q), f)$.

To elicit our stability results, we need the following definition of $\eps$-mild equilibria.
\begin{Definition}[$\eps$-mild equilibrium]\label{def.equi.eps}
	Let $\varepsilon\geq 0$. A closed set $S\subset \X$ is called an $\varepsilon$-mild equilibrium (w.r.t. $f$ and $Q$), if 
\be	\label{eq.def.epsout}f(x)\leq J^Q(x,S,f)+\eps \quad \forall x\notin S. \ee
Denote by $\Ec^Q(f,\varepsilon)$ the set of $\varepsilon$-mild equilibria w.r.t. $(f,Q)$.
When $\varepsilon=0$, we still call $S$ a mild equilibrium and may use the notation $\Ec^Q(f)$ instead of $\Ec^Q(f,0)$.
\end{Definition}	
We also denote
$$
V^Q_\varepsilon(x,f):=\sup_{S\in \Ec^Q(f,\varepsilon)} J^Q(x,S,f)\quad \forall x\in \X,
$$
and we keep using the notation $V^Q(x,f)$ in \eqref{eq.def.V} instead of $V^Q_0(x,f)$ when $\eps=0$.

\section{Main results}\label{sec:stability}

Consider a sequence $(f^n, Q^n)_{n\in \overline \N}$, where $(f^n)_{n\in \overline \N}$ are reward functions, and $(Q^n=(\mu^n, \sigma^n))_{n\in\overline \N}$ are regular coupled functions such that, for each $n\in \overline\N$, $X$ governed by
$$
dX_t=\mu^n(X_t)dt+\sigma^n(X_t)dW_t,
$$
is supported on $\X$ for any $X_0\in\X$.
\begin{Theorem}\label{thm.continue.eps}
	Suppose Assumption \ref{assume.delta.weight} and the following hold:
\bi 
{\color{black}\item[(i)] $(Q^n)_{n\in \overline \N}$ are regular, and satisfy
	\be\label{eq.mu.sigma}
	\sup_{n\in \overline \N}\left(\|\mu^n\|_\infty+\|\sigma^n\|_\infty\right)<\infty \quad\text{and}\quad \inf_{n\in \overline \N, x\in \X} |\sigma^n(x)|^2=:L>0;
	\ee
	\item[(ii)]  $f^n\geq 0$ is continuous for any $n\in \overline\N$, and $\|f^\infty\|_{\infty}+\sup_{x,y\in \X}\frac{|f^\infty(x)-f^\infty(y)|}{|x-y|}=:K<\infty$;}
\item[(iii)] 
$
\|\mu^n-\mu^\infty\|_{\infty}+\|\sigma^n-\sigma^\infty\|_{\infty}+\|f^n-f^\infty\|_{\infty}\rightarrow 0,
$ as $n\to \infty$.
\ei 
Then
\begin{align*}
	\lim\limits_{\varepsilon\searrow 0}\Big( \liminf_{n\rightarrow \infty} V^{Q^n}_\varepsilon(x,f^n)\Big)=
	\lim\limits_{\varepsilon\searrow 0}\Big( \limsup_{n\rightarrow \infty}V^{Q^n}_\varepsilon(x,f^n)\Big)=V^{Q^\infty}(x,f^\infty)\quad \forall x\in \X.
\end{align*}	
\end{Theorem}


\begin{Theorem}\label{thm.semi.continue}
	Suppose the assumptions in Theorem \ref{thm.continue.eps} hold. Then
	\begin{equation}\label{e2}
		\limsup_{n\to \infty} V^{Q^n}(x,f^n)\leq V^{Q^\infty}(x, f^\infty)\quad \forall x\in \X.
	\end{equation}
\end{Theorem}

\begin{Remark}
	Exact continuity in \eqref{e2} may fail in general. See the example in Section 4.1.
\end{Remark}

\begin{Remark}
	By an argument similar to that in \cite[Remark 4.2]{bayraktar2022stability}, we have that\footnote{The lower/upper limit of a sequence of sets is defined in a usual way. That is, for a sequence of sets $(A_n)_{n\in \overline\N}$, 
	$$\liminf_{n\to\infty} A_n:=\underset{n\in \N}{\cup} \underset{k\geq n}{\cap} A_k\quad\text{and}\quad \limsup_{n\to\infty} A_n:=\underset{n\in \N}{\cap} \underset{k\geq n}{\cup}  A_k.$$
	We say the sequence of sets $(A_n)_{n\in \overline\N}$ is lower (resp. upper) semi-continuous if $A_\infty\subset \liminf_{n\to\infty} A_n$ (resp. $A_\infty\supset\limsup_{n\to\infty} A_n$).}
	$$
	\Ec^{Q^\infty}(f^\infty)=\lim_{\eps \searrow 0} \Big( \liminf_{n\rightarrow \infty} \Ec^{Q^n}_\eps (f^n)  \Big)= \lim_{\eps \searrow 0} \Big( \limsup_{n\rightarrow \infty} \Ec^{Q^n}_\eps (f^n)  \Big).
	$$
\end{Remark}

\subsection{Proofs of Theorems \ref{thm.continue.eps} and \ref{thm.semi.continue}}
To begin with, we first fix an arbitrary $(f,Q)$ and study the relation between $V^Q_\eps(\cdot, f)$ and $V^Q(\cdot, f)$.

\begin{Proposition}\label{prop0}
	Suppose that $f$ is continuous with $\|f\|_\infty<\infty$ and $Q=(\mu,\sigma)$ is regular. Then
\be\label{eq.converge.eps} 
V^Q(x,f)=\lim_{\eps \searrow 0}V^Q_\eps(x,f)\quad \forall x\in \X. 
\ee 
\end{Proposition}

\begin{proof}
We prove \eqref{eq.converge.eps} by contradiction. For any $\varepsilon>0$, $\Ec^Q(f)\subset \Ec_{\eps}^Q(f)$ implies that $V^Q(\cdot ,f)\leq V^Q_\varepsilon(\cdot, f)$. Suppose there exists $x_0\in \X$ such that 
\be\label{eq.contradict.0}  
\limsup_{\eps\searrow 0}V^Q_\eps(x_0, f)-V^Q(x_0, f)=\alpha>0.
\ee 
Then there exists a sequence $(\eps_k, S_k)_{k\in \N}$ such that $\eps_k\searrow 0$, $S_k\in \Ec_{\eps_k}^Q(f)$ are closed, and 
\be\label{eq.contradict.1} 
J^Q(x_0, S_k, f)-V^Q(x_0, f)\geq \frac{\alpha}{2} \quad \forall k\in \N.
\ee 
For any $k\in \N$, we have $x_0\notin S_k$, for otherwise $J^Q(x_0, S_k,f)=f(x_0)\leq V^Q(x_0, f)$, which contradicts \eqref{eq.contradict.1}. Define
$$
l_k:=\sup\{y<x_0: y\in S_k\},\quad r_k:=\inf\{y>x_0: y\in S_k\},\quad \forall k\in \N.
$$
Now consider the sequence $(l_k)_{k\in \N}$. If $(l_k)_{k\in \N}$ is bounded, then we take a subsequence $(l_{k_j})_{j\in \N}$ such that $\lim_{j\to\infty}l_{k_j}=l$ for some constant $l\leq x_0$. Otherwise, we take a subsequence $(l_{k_j})_{j\in \N}$ that tends to $l=-\infty$. Similarly, for the subsequence $(r_{k_j})_{j\in \N}$, find a further subsequence, which we still denote as $(r_{k_j})_{j\in \N}$, such that $r_{k_j}$ either converges to a constant or tends to $\infty$, and we use $r$ to denote the limit no matter which case it is. Hence, we find a sequence of intervals $((l_{k_j}, r_{k_j}))_{j\in \N}$ that converges to interval $(l,r)$. Notice that $(l_{k_j})_{j\in \N}, (r_{k_j})_{j\in \N}$ can be chosen to be monotone, so for any $y\in (l,r)$, 
$$
y\in (l_{k_j}, r_{k_j}),  \; \text{and}\; J^Q(y, S_{k_j},f)=J^Q(y, \X\setminus (l_{k_j}, r_{k_j}),f), \text{for $j$ large enough.}
$$
Now fix $y\in (l,r)$. By $S_{k_j}\in \Ec_{\eps_{k_j}}^Q(f)$ and the dominated convergence theorem, we have that
$$
J^Q(y, \X\setminus (l,r),f)=\lim_{j\to\infty } \left(J^Q(y, \X\setminus (l_{k_j}, r_{k_j}))+\eps_{k_j} \right)\geq f(y)\quad \forall y\in (l,r).
$$
Hence, $\X\setminus (l,r)\in \Ec^Q(f)$. Then it follows from $l\leq x_0\leq r$ that
$$
V^Q(x_0,f)\geq J^Q(x_0, \X\setminus (l,r),f)=\lim_{j\to \infty} J^Q(x_0, \X\setminus (l_{k_j}, r_{k_j}),f)=\lim_{j\to \infty} V^Q_{\eps_{k_j}}(x, f),
$$
which contradicts \eqref{eq.contradict.0}. 
\end{proof}

Next, let us go back to the sequence $((f^n, Q^n))_{n\in \overline \N}$, and introduce the following Lemma.
\begin{Lemma}\label{lm}
	Suppose the assumptions in Theorem \ref{thm.continue.eps} hold. Then
	$$
	\lim_{n\rightarrow\infty} \sup_{x\in \X, S\in \Bc} |J^{Q^n}(x,S,f^n)-J^{Q^\infty}(x,S,f^\infty)|=0.
	$$
\end{Lemma}

\begin{proof}
	By assumptions, for any $S\in \Bc$,
\begin{align*}
	|J^{Q^n}(x,S,f^n)-J^{Q^\infty}(x,S,f^\infty)|\leq & \|f^n-f^\infty\|_{\infty}+|J^{Q^n}(x,S,f^\infty)-J^{Q^\infty}(x,S,f^\infty)|.
\end{align*}
To prove the desired result, it is sufficient to show the convergence of the second term above. To this end, fix $\eps>0$ and we will find $N$ such that 
\be\label{eq.lm.unifconverge} 
\sup_{x\in \X,S\in \Bc}|J^{Q^n}(x,S,f^\infty)-J^{Q^\infty}(x,S,f^\infty)|\leq \eps\quad \forall n\geq N.
\ee

Take an arbitrary $S\in \Bc$. For each $n\in \overline \N$ and $r\in(0,\infty)$, set $v_r^n(x):=\E^{Q^n}_x[e^{-r \rho_S} f^\infty(X_{\rho_S})]$. Recall the cumulative function $F(r)$ in Assumption \ref{assume.delta.weight} and the constants $K, L$ in the assumptions
of Theorem \ref{thm.continue.eps}. As $\lim_{t\to\infty}\delta(t)=0$, $F(0)=0$. By the right-continuity of function $F$,
there exists $r_0>0$ such that 
\be\label{eq.lm.r} 
F(r_0)\leq \frac{\eps}{4\|f^\infty\|_{\infty}}.
\ee
We proceed with the rest of the proof in three steps. 

\textbf{Step 1.} We first focus on the case $0<r\leq r_0$. Notice that $\|v_r^n\|_\infty\leq \|f^\infty\|_{\infty}$ for any $n\in \overline \N$, then by \eqref{eq.lm.r}, 
\be\label{eq.r.small} 
\int_{[0,{r_0}]}| v^n_r-v^\infty_r| F(dr)\leq \left( \sup_{0<r\leq r_0, n\in \overline \N} 2\|v^n_r\|_{\infty}  \right) \cdot F(r_0)\leq 2 \|f^\infty\|_{\infty}\cdot \frac{\eps}{4\|f^\infty\|_{\infty}}=\frac{\eps}{2}.
\ee 

\textbf{Step 2.} Pick an arbitrary $r> r_0$. We first construct a bound for $\sup_{x\in \X\setminus \overline{S}} |(v^n_r)'|+|(v^n_r)''|$. 
The Lipschitz continuity and boundedness of $\mu^n, \sigma^n$ imply H\"older continuity. Then, given an interval $(a,b)\subset \X\setminus \overline S$ with $0<b-a\leq 1$, {\color{black}it is known (see, e.g. \cite[Theorem 9.2.14]{MR2001996})} that $v^n_r(x)$ is twice continuously differentiable and satisfies
\be\label{eq.r.elliptic} 
	-rv^n_r(x)+\mu^n(x)(v^n_r(x))'+\frac{1}{2}(\sigma^n)^2(x)(v^n_r(x))''=0 \quad x\in (a,b).
\ee 
Write $\X\setminus \overline{S}=: \cup_{i\in \N} (\theta_i, \theta_{i+1})$. Take an arbitrary $x_0\in \X\setminus \overline{S}$. There are two cases: (I) for the case $x_0\in (\theta_i, \theta_{i+1})$ with $\theta_{i+1}-\theta_i\leq 1$, we set $a=\theta_i, b=\theta_{i+1}$; (II) for the case $x_0\in (\theta_i, \theta_{i+1})$ with $\theta_{i+1}-\theta_i>1$ (notice that $\theta_i, \theta_{i+1}$ can be $-\infty, \infty$ respectively), we take $a,b$ such that $\theta_i\leq a<x_0<b\leq  \theta_{i+1}$ with $b-a=1$. Let $v^n_r(x)=\tilde{u}(\phi(x))$, where $\phi(x):=\int_0^x\exp(-\int_0^l \frac{2\mu^n(z)}{(\sigma^n)^2(z)}dz)dl$. Then \eqref{eq.r.elliptic} leads to 
\be\label{eq.tildeu.elliptic} 
\begin{cases}
	-r\tilde{u}(y)+\frac{1}{2}\tilde{\sigma}^2(y)\tilde{u}''(y)=0 \quad y\in (\phi(a),\phi(b)),\\
	\tilde{u}(\phi(a))=v^n_r(a), \tilde{u}(\phi(b))=v^n_r(b),
\end{cases}
\ee
where $\tilde{\mu}(y)=\mu^n(\phi^{-1}(y))$ and $\tilde{\sigma}(y)=\sigma^n(\phi^{-1}(y))\phi'(\phi^{-1}(y))$. The boundedness of $\sup_n \|v^n_r\|_\infty$ gives the uniform boundedness of $\tilde{u}$ over $n$. Then the first line in \eqref{eq.tildeu.elliptic} together with \eqref{eq.mu.sigma} gives that
$
|\tilde{u}''|\leq \frac{2r K}{L}\ \text{on}\; (\phi(a), \phi(b)).
$	
Then a direct calculation along with $|b-a|\leq 1$ and the uniform boundedness of $(Q^n)_{n\in\overline \N}$ shows that
\be\label{eq.lm.phi} 
\sup_{x\in [a,b]} \Big( |\phi'(x)|+|\phi''(x)|+|\phi'(x)|^{-1} \Big) :=M<\infty,
\ee
where the constant $M$ depends on $K,L$ but does not depend on $n,r,S$.
By Mean Value Theorem, for both cases (I)\&(II), the second line in \eqref{eq.tildeu.elliptic} together with \eqref{eq.lm.phi} gives that 
$$
\begin{aligned}
	|\tilde u'({y_0})|=\left|\frac{v^n_r(b)-v^n_r(a)}{b-a}\cdot \frac{b-a}{\phi(b)-\phi(a)}\right|
	\leq \max \left\{ \sup_{x,y\in \X}\frac{|f^\infty(x)-f^\infty(y)|}{|x-y|}, \frac{2 \|f\|_\infty}{1} \right\}M
	\leq 2KM,
\end{aligned}
$$ 
for some point $y_0\in (\phi(a), \phi(b))$. Hence, as $x_0\in (a,b)$ and $|b-a|\leq 1$,
$$
|\tilde{u}'(\phi(x_0))|\leq 	|\tilde u'({y_0})|+|\int_{y_0}^{\phi(x_0)} \tilde{u}''(y)dy  |\leq 2KM+\int_0^{M} \frac{2r K}{L}dl\leq \widetilde{M}(1+r), 
$$
{\color{black} where $\widetilde{M}$ is a constant that depends on $K,L$ but does not depend on $n,r,S$, and may change from line to line during the rest of the proof.} In sum, $|\tilde{u}'(\phi(x_0))|+|\tilde{u}''(\phi(x_0))|\leq \widetilde{M}(1+r)$, then by \eqref{eq.lm.phi} again, $|(v^n_r)'(x_0)|+|(v^n_r)''(x_0)|\leq \widetilde{M}(1+r)$.
As $x_0$ is arbitrary, 
\be\label{eq.lm.3}  
\sup_{x\in \X\setminus \overline{S}, n\in \overline \N} |(v^n_r)'(x)|+|(v^n_r)''(x)|\leq \widetilde{M}(1+r).
\ee 	
Now we estimate $|v^\infty_r(\cdot)- v^n_r(\cdot)|$ for $r\geq r_0$. Take $n<\infty$ and set $\bar{v}_r:= v^\infty_r-v^n_r$. Since
$$
-r v_r^n(x)+\mu^n(x) (v_r^n)' (x)+\frac{1}{2}(\sigma^n(x))^2 (v^n_r)''(x)=0\quad \text{on}\; \X\setminus \overline{S},\quad \forall n\in \overline \N,
$$
we have that
\be\label{eq.lm.1} 
-r \bar{v}_r(x)+\mu^\infty(x) (\bar{v}_r)'(x)+\frac{1}{2}(\sigma^\infty(x))^2 (\bar{v}_r)''(x)+g(x)=0\quad \text{on } \X\setminus \overline{S},
\ee 
where
\be\label{eq.lm.1'} 
g(x):=(\mu^\infty(x)-\mu^n(x)) (v^n_r)'(x)+\frac{1}{2}[(\sigma^\infty(x))^2-(\sigma^n(x))^2] (v^n_r)''(x).
\ee 
Meanwhile, $\bar{v}_r$ in \eqref{eq.lm.1} has the following probabilistic representation
\be\label{eq.lm} 
\bar{v}_r(x)=\E^{Q^\infty}_x\left[\int_0^{\rho_{X\setminus \overline{S}}} e^{-rs}g(X_s)ds+ e^{-r \rho_{X\setminus \overline{S}}} \bar{v}_r(X_{\rho_S})\right]\quad \forall x\in \X\setminus \overline{S}.
\ee 
By \eqref{eq.lm.3} and \eqref{eq.lm.1'},
\be\label{eq.lemma} 
|g(x)|\leq \widetilde{M}(1+r)[\|\mu^n-\mu^\infty\|_{\infty}+\|\sigma^n-\sigma^\infty\|_{\infty}] \quad \forall x\in \X\setminus \overline{S}.
\ee 
Notice that $\bar{v}\mid _{\partial S}=0$ and $r\geq r_0>0$, then from \eqref{eq.lm} and \eqref{eq.lemma} we deduce that
\be\label{eq.r.big} 
\begin{aligned}
	|\bar{v}_r(x)|\leq \sup_{y\in \X\setminus \overline{S}}|g(y)|\int_0^\infty e^{-rt}dt
	\leq \widetilde{M}\left(1+\frac{1}{r_0}\right) [\|\mu^n-\mu^\infty\|_{\infty}+\|\sigma^n-\sigma^\infty\|_{\infty}]\quad \forall x\in \X\setminus \overline{S}.
\end{aligned}
\ee 

\textbf{Step 3.} Since the RHS of  \eqref{eq.r.big} is independent of $S$, we can now choose $N$ independent of $S$ such that the RHS of \eqref{eq.r.big} is less than $\frac{\eps}{2}$ for any $n>N$. This together with \eqref{eq.r.small} implies that
$$
\begin{aligned}
	&\sup_{x\in \X}\left|J^{Q^n}(x,S,f^\infty)-J^{Q^\infty}(x,S,f^\infty)\right|=\sup_{x\in \X\setminus \overline{S}}\left|J^{Q^n}(x,S,f^\infty)-J^{Q^\infty}(x,S,f^\infty)\right|\\
	=&\sup_{x\in \X\setminus \overline{S}} \left|\int_0^\infty v^n_r(x)F(dr)-\int_0^\infty v^\infty_r(x) F(dr) \right|\leq\int_{[0,{r_0}]} \|\bar{v}_r\|_{\infty} F(dr)+\int_{(r_0,{\infty})} \|\bar{v}_r\|_{\infty} F(dr)\leq\eps.
\end{aligned}
$$
where the first line follows from the fact that $J^{Q^n}(x,S,f^\infty)=f^\infty(x)=J^{Q^\infty}(x,S,f^\infty)$ for $x\in \overline{S}$.
\end{proof}

We are now ready for the proofs of Theorems \ref{thm.continue.eps} and  \ref{thm.semi.continue}.
\begin{proof}[{\bf Proof of Theorem \ref{thm.continue.eps}}]
	By Lemma \ref{lm.optimalmild}, for any $n\in \overline \N$,
$$
V^{Q^n}(x,f^n)=J^{Q^n}(x, S^*(f^n,Q^n), f^n)\quad \forall x\in \X.
$$
Now we divide the proof into two steps.

{\bf Step 1.} We first prove that, for any $\varepsilon>0$,
\be\label{eq.continue.upper}  
\limsup_{n\rightarrow \infty} V^{Q^n}_{\eps/2}(x,f^n) \leq V^{Q^\infty}_\eps(x,f^\infty)\quad \forall x\in \X.
\ee 
Fix an arbitrary $\eps>0$. By Lemma \ref{lm} and $\|f^n-f^\infty\|\to 0$, there exists $N$ such that for all $n\geq N$, 
$$
\sup_{x\in \X,S\in \Bc} |J^{Q^n}(x,S,f^n)-J^{Q^\infty}(x,S,f^\infty)|\leq \frac{\eps}{2}\quad \text{and}\quad f^\infty\leq f^n+\frac{\eps}{2}.
$$
Then for any $n\geq N$ and $S\in \Ec_{\eps/2}^{Q^n}(f^n)$, we have that 
$$
J^{Q^\infty}(x,S,f^\infty)\geq J^{Q^n}(x, S,f^n)-\frac{\eps}{2}\geq f^n(x)-\frac{\eps}{2}\geq  f^\infty(x)-\eps\quad \forall x\notin S,
$$
which implies that $S\in \Ec_{\eps/2}^{Q^\infty}(f^\infty)$. Hence,
$\Ec_{\eps/2}^{Q^n}(f^n)\subset \Ec_{\eps}^{Q^\infty}(f^\infty)$ for $n\geq N$.
Then for any $x\in \X$
$$
\begin{aligned}
	\limsup_{n\to\infty} V^{Q^n}_{\eps/2}(x,f^n)\leq  &\limsup_{n\to\infty} \sup_{S\in \Ec_\eps^{Q^\infty}(f^\infty)} J^{Q^n}(x,S,f^n)\\
	=& \sup_{S\in \Ec_\eps^{Q^\infty}(f^\infty)} J^{Q^\infty}(x,S,f^\infty)=V_\eps^{Q^\infty}(x,f^\infty),
\end{aligned}
$$
where the first equality follows from Lemma \ref{lm}, so \eqref {eq.continue.upper} is established.

{\bf Step 2.} Now we prove the desired result. By \eqref{eq.continue.upper} and Proposition \ref{prop0},
\be\label{eq.eps.sup} 
\begin{aligned}
	\lim\limits_{\varepsilon\searrow 0}\Big( \limsup_{n\rightarrow \infty}V^{Q^n}_\varepsilon(x,f^n)\Big)
	\leq  \limsup\limits_{\varepsilon\searrow 0} V^{Q^\infty}_{2\eps}(x,f^\infty)
	=V^{Q^\infty}(x,f^\infty)\quad \forall x\in \X.
\end{aligned}	
\ee 
In addition, for any $\eps>0$, by Lemma \ref{lm}, for $n$ large enough,
\be\label{eq.thm1.1} 
\begin{aligned}
	J^{Q^n}(x, S^*(f^\infty,Q^\infty), f^n)\geq &J^{Q^\infty}(x, S^*(f^\infty,Q^\infty), f^\infty)-\frac{\eps}{2}\\
	\geq &f^\infty(x)-\frac{\eps}{2}\geq f^n-\eps\quad \forall x\notin S^*(f^\infty,Q^\infty),
\end{aligned}
\ee
so $S^*(f^\infty,Q^\infty)\in \Ec_{\eps}^{Q^n}(f^n)$ for $n$ large enough. Therefore,
$$
V^{Q^\infty}(x,f^\infty)=J^{Q^\infty}(x, S^*(f^\infty,Q^\infty), f^\infty)\leq\liminf_{n\rightarrow \infty} V^{Q^n}_\varepsilon(x,f^n)\quad \forall x\in \X, 
$$
which implies that
\be\label{eq.eps.inf} 
V^{Q^\infty}(x,f^\infty)\leq \lim\limits_{\varepsilon\searrow 0}\Big( \liminf_{n\rightarrow \infty} V^{Q^n}_\varepsilon(x,f^n)\Big)\quad \forall x\in \X.
\ee 
Then the desired result follows from \eqref{eq.eps.sup} and \eqref{eq.eps.inf}.	
\end{proof}

\begin{proof}[{\bf Proof of Theorem \ref{thm.semi.continue}}]
	Take an arbitrary $\eps>0$. Lemma \ref{lm} enables us to exchange $n$ and $\infty$ in \eqref{eq.thm1.1} to conclude that
$
S^*(f^n,Q^n)\in \Ec_\eps(f^\infty, Q^\infty),
$ for $n$ large enough.
This together with Lemma \ref{lm.optimalmild} implies that
$$
V^{Q^n}(x,f^n)=J^{Q^n}(x, S^*(f^n,Q^n), f^n)\leq \sup_{S\in \Ec_\eps(f^\infty, Q^\infty)} J^{Q^n}(x, S, f^n), \quad  \text{for $n$ large enough}.
$$
By applying Lemma \ref{lm} again to above inequality, we have that for $x\in \X$,
\begin{align*}
	\limsup_{n\to\infty}V^{Q^n}(x,f^n)\leq &\lim_{n\to\infty}\sup_{S\in \Ec_\eps(f^\infty, Q^\infty)} J^{Q^n}(x, S, f^n)\\
	=& \sup_{S\in \Ec_\eps(f^\infty, Q^\infty)} J^{Q^\infty}(x, S, f^\infty)=V^{Q^\infty}_\eps(x,f^\infty).
\end{align*}
Then by letting $\eps\searrow 0$ and applying Proposition \ref{prop0}, we achieve \eqref{e2}.
\end{proof}

\section{Examples}\label{sec:example}
In Section \ref{eg1}, we provide an example where $\limsup_{n\to\infty} V^{Q^n}(x,f^n)< V^{Q^\infty}(x,f^\infty)$.
This indicates that the exact continuity for $(f^n,Q^n)\mapsto V(f^n,Q^n)$ may fail, which further justifies the necessity of use for $\eps$-mild equilibria for the value function.  In Section \ref{eg2}, we present an example where $S^*(f^\infty, Q^\infty)\nsubseteq\liminf_{n\to\infty} S^*(f^n, Q^n)$. In particular, this contrasts with \cite[Theorem 3.1]{bayraktar2022stability}, the lower semi-continuity for $(f^n,Q^n)\mapsto S^*(f^n,Q^n)$ in the discrete-time context. Throughout this section, {\color{black}$g'(x-)$ (resp. $g'(x+)$) denotes the left (resp. right) derivative.}

\subsection{An example of strict upper semi-continuity}\label{eg1}
Let $\X=\R$, $\delta(t)=\frac{1}{1+\beta t}$ with $\beta>0$, and set
$ \mu^n\equiv -\frac{1}{n}, \ \sigma^n\equiv 1$ for all $n\in\overline\N$. We have that
\be\label{eq.eg.delta} 
\delta(t+s)> \delta(t)\delta(s)\quad \forall t,s>0.
\ee 
We choose arbitrary constants $a,b,d\in\R$ with $a<b$ and $d>0$, and define 
$$
\begin{cases}
	J_b(x):=d\E^{Q^\infty}_x[\delta(\rho_{\{b\}})]\quad \forall x\in \R;\quad  c:= J_b(a)>0;\\
	J_{a,b}(x):=c\E_x^{Q^\infty}[\delta(\rho_{\{a,b\}})\cdot 1_{\{\rho_{\{a,b\}=a}\}}]+d\E^x[\delta(\rho_{\{a,b\}})\cdot 1_{\{\rho_{\{a,b\}=b}\}}\quad \forall x\in \R.
\end{cases}
$$ 
Define reward functions $f^n\equiv f$ for all $n\in \overline \N$ with
\be\label{eq.eg.f} 
	f(x) :=
\begin{cases}
	\begin{aligned}
&e^{-2(a-x)}\cdot J_{a,b}(x)\quad &  x\in (-\infty, a)\\
&\frac{1}{1+L_0(x-a)(b-x)} \cdot  J_{a,b}(x) & \quad x\in[a,b]\\
&J_{b}(x) & \quad x\in (b,\infty)
	\end{aligned}
\end{cases},
\ee 
where $L_0>0$ is a constant. We first provide formulas of $J_b, J_{a,b}$ and relations of $J_b, J_{a,b}$ and $f$ as follows.

\begin{Lemma}\label{lm.eg}
We have that
\begin{align}
	\notag J_b(x)
 &=  d\int_0^\infty e^{-s}e^{-|x-b|\sqrt{2\beta s}}ds\quad\forall x\in\R;\\
\label{eq.eg.jab}  	J_{a,b}(x)&=	\begin{cases}	
	c\int_0^\infty e^{-s}e^{-|x-a|\sqrt{2\beta s}}ds,&x\in (-\infty, a)\\
	c\int_0^\infty e^{-s}\frac{\sinh((b-x)\sqrt{2\beta s})}{\sinh((b-a)\sqrt{2\beta s})}ds+  d\int_0^\infty e^{-s}\frac{\sinh((x-a)\sqrt{2\beta s})}{\sinh((b-a)\sqrt{2\beta s})}ds &x\in [a,b],\\
	J_b(x) &x\in (b,\infty).
\end{cases}
\end{align}
Moreover, 
\be\label{eq.eg.fjb} 
\begin{cases}
	\begin{aligned}
		& J_b(x)=J_{a,b}(x)=f(x) & x\in \{a\}\cup [b,\infty);\\
		&J_b(x)>J_{a,b}(x)>f(x) & x\in (-\infty,a)\cup (a,b),
	\end{aligned}
\end{cases}
\ee
as shown in Figure \ref{figure}.
\end{Lemma}

\begin{proof}
By the definition of $J_b$, a direct calculation shows
\be\label{eq.eg.jblm}  
\begin{aligned}
J_b(x)&=d\E^{Q^\infty}_x[\delta(\rho_{\{b\}})]=d\int_0^\infty \frac{p_b(t)}{1+\beta t} dt=d\int_0^\infty \int_0^\infty e^{-(1+\beta t)s}p_b(t)dsdt\\
 &=d\int_0^\infty e^s \left( \int_0^\infty  e^{-\beta ts}p_b(t)dt\right)ds=d\int_0^\infty e^{-s}\E^{Q^\infty}_x[e^{-\beta s \rho_{\{b\}}}]ds\\
&=  d\int_0^\infty e^{-s}e^{-|x-b|\sqrt{2\beta s}}ds\quad\forall x\in\R,
\end{aligned}
\ee 
where $p_b(t)$ denotes the density function of $\rho_{\{b\}}$ under $\P^{Q^\infty}_x$, and the third line follows from the formula \cite[2.0.1 on page 204]{MR1912205}. Similarly,
\be\label{eq.eg.jablm} 
\begin{aligned}
	J_{a,b}(x)&:=c\E^{Q^\infty}_x[\delta(\rho_{\{a,b\}})\cdot 1_{\{\rho_{\{a,b\}=a}\}}]+d\E^{Q^\infty}_x[\delta(\rho_{\{a,b\}})\cdot 1_{\{\rho_{\{a,b\}=b}\}}]\\
& = c\int_0^\infty \int_0^\infty e^{-(1+\beta t)s}p_{a,b}(t)dsdt+d\int_0^\infty \int_0^\infty e^{-(1+\beta t)s}q_{a,b}(t)dsdt\\
 &=c\int_0^\infty e^s \left( \int_0^\infty  e^{-\beta ts}p_{a,b}(t)dt\right)ds+d\int_0^\infty e^s \left( \int_0^\infty  e^{-\beta ts}q_{a,b}(t)dt\right)ds\\
	& = c\int_0^\infty e^{-s}\E^x[e^{-\beta s \rho_{\{a,b\}}}\cdot 1_{\{X_{\rho_{\{a,b\}}=a}\}}]ds +d\int_0^\infty e^{-s}\E^x[e^{-\beta s \rho_{\{a,b\}}}\cdot 1_{\{X_{\rho_{\{a,b\}}=b}\}}]ds\\
 &=	\begin{cases}	
		c\int_0^\infty e^{-s}e^{-|x-a|\sqrt{2\beta s}}ds,&x\in (-\infty, a)\\
		c\int_0^\infty e^{-s}\frac{\sinh((b-x)\sqrt{2\beta s})}{\sinh((b-a)\sqrt{2\beta s})}ds+  d\int_0^\infty e^{-s}\frac{\sinh((x-a)\sqrt{2\beta s})}{\sinh((b-a)\sqrt{2\beta s})}ds &x\in [a,b],\\
		J_b(x) &x\in (b,\infty).
	\end{cases}
\end{aligned}
\ee 
where $p_{a,b}(t)$ (resp. $q_{a,b}(t)$) denotes the density of $\rho_{\{a,b\}}$ on $\{X_{\rho_{\{a,b\}}}=a\}$ (resp. $\{X_{\rho_{\{a,b\}}}=b\}$) under $\P^{Q^\infty}_x$, and the last line above follows from formulas \cite[3.0.5 (a)\&(b) on page 218]{MR1912205}. 
%

{\color{black}We have that for $x\in (a,b)$,
\begin{equation}\label{eq.eg.prepar}
\begin{aligned}
J_{a,b}(x)-J_{b}(x)&=\E^{Q^\infty}_x\left[1_{\{X_{\rho_{\{a,b\}}=a}\}}\left(c\delta(\rho_{\{a,b\}}-d\delta(\rho_{\{b\}})\right)\right]\\
&<\E^{Q^\infty}_x\left[1_{\{X_{\rho_{\{a,b\}}=a}\}}\delta(\rho_{\{a,b\}})\left(c-d\delta(\rho_{\{b\}}-\rho_{\{a,b\}})\right)\right]\\
&=\E^{Q^\infty}_x\left[1_{\{X_{\rho_{\{a,b\}}=a}\}}\delta(\rho_{\{a,b\}})\left(c-d\E^{Q^\infty}_x\left[\delta(\rho_{\{b\}}-\rho_{\{a,b\}})\Big|\mathcal{F}_{\rho_{\{a,b\}}}\right]\right)\right]\\
&=\E^{Q^\infty}_x\left[1_{\{X_{\rho_{\{a,b\}}=a}\}}\delta(\rho_{\{a,b\}})\left(c-J_b(a)\right)\right]=0,\\
\end{aligned}
\end{equation}}
where the second line follows from \eqref{eq.eg.delta} and $f(b)=d>0$. Similarly, we have
\be\label{eq.eg.prepar'}  
J_{a,b}(x)< J_b(x)\quad \forall x\in(-\infty, a).
\ee 
Then, by combining \eqref{eq.eg.prepar}, \eqref{eq.eg.prepar'} and \eqref{eq.eg.f}, we reach to \eqref{eq.eg.fjb}.
\end{proof}

Notice that all the conditions in Theorem \ref{thm.continue.eps} are satisfied. The following proposition shows that the strict semi-continuity in \eqref{e2} holds in this example.

\begin{figure}[h!]
	\centering
	\includegraphics[width=15cm]{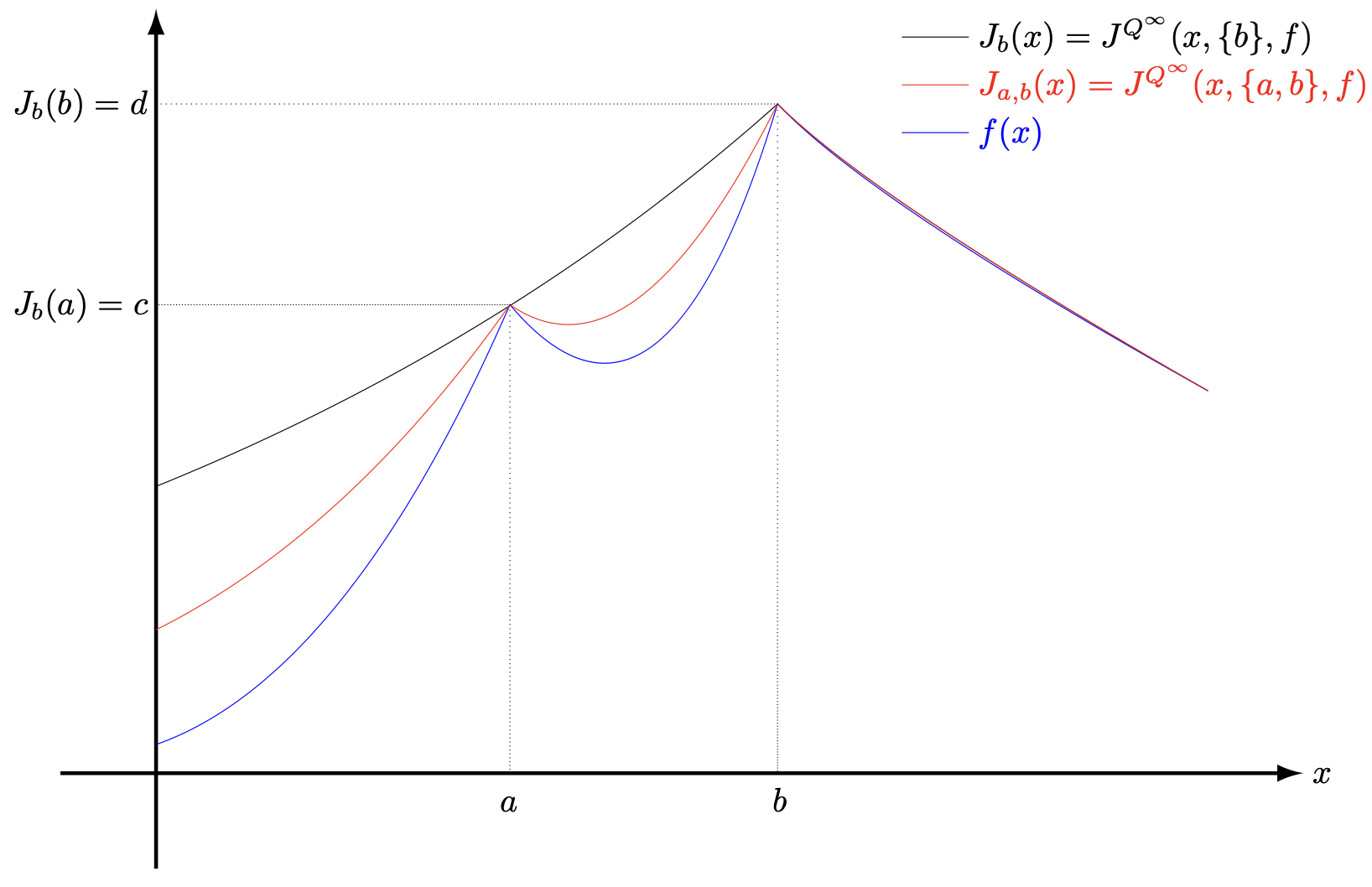}
	\caption{Relations of $J_b(\cdot), J_{a,b}(\cdot)$ and $f(\cdot)$}\label{figure}
	\centering
\end{figure}

\begin{Proposition}\label{prop.eg1}
We have that $S^*(Q^\infty, f)=\{b\}$, and $S^*(f, Q^n)=\{a,b\}$ for $n\in \N$ large enough. 
	Moreover,
	\be\label{eq.eg.} 
	\limsup_{n\to\infty} V^{Q^n}(x, f)<V^{Q^\infty}(x, f)\quad \forall x\in(-\infty, a).
	\ee 
\end{Proposition}

\begin{proof}

	By construction of $f$ in \eqref{eq.eg.f}, we can see that $f>0$ and $f(b)=\max_{x\in \R} f(x)$,  thus $b$ must be contained in $S^*(f, Q^n)$ for any $n\in\overline\N$.
	Since $f(x)\leq J_b(x)=J^{Q^\infty}(x, f, \{b\})$, we have $S^*(f, Q^\infty)=\{b\}$. 
	
	Now we show that for $n\in\N$ large enough, $\{a,b\}$ is a mild equilibrium for $Q^n$. For all $n\in \N$, write $J^n_b(x):=J^{Q^n}(x, \{b\}, f)$ and $J_{a,b}^n(x):=J^{Q^n}(x,\{a,b\},f)$ for short. First, by applying similar arguments in \eqref{eq.eg.jblm}, \eqref{eq.eg.jablm} with formulas \cite[2.0.1 on page 301, 3.0.5 (a)\&(b) on page 315]{MR1912205}, 
	we have that, for any $n\in \N$,
	\be\label{eq.eg.jn} 
	\begin{aligned}
	&J^n_b(x)=d\int_0^\infty e^{-s}e^{-\frac{1}{n}(b-x)- |b-x|\sqrt{2\beta s+\frac{1}{n^2}}}ds\quad \forall x\in \R ;\\
	&J^n_{a,b}(x)=
	\begin{cases}
	c\int_0^\infty e^{-s}e^{-\frac{1}{n}(a-x)- |a-x|\sqrt{2\beta s+\frac{1}{n^2}}}ds  & x\in(-\infty,a)\\
		\int_0^\infty  \left(c e^{\frac{1}{n}(x-a)}e^{-s}\frac{\sinh((b-x)\sqrt{2\beta s+\frac{1}{n^2}})}{\sinh((b-a)\sqrt{2\beta s +\frac{1}{n^2}})}+ d  e^{\frac{1}{n}(x-b)}e^{-s}\frac{\sinh((x-a)\sqrt{2\beta s +\frac{1}{n^2}})}{\sinh((b-a)\sqrt{2\beta s +\frac{1}{n^2}})}\right)ds,&x\in [a,b]\\
	J^{Q^n}(x,\{b\}, f) &x\in(b,\infty)
	\end{cases}.
	\end{aligned} \quad
	\ee
	By the formulas of $J^n_b$ in \eqref{eq.eg.jn} and the relation $\sqrt{2\beta s+\frac{1}{n^2}}<\sqrt{2\beta s}+\frac{1}{n}$, we have that for any $n\in \N$,
	\be\label{eq.eg1.0}  
	\begin{cases}
		J_b(x)\cdot e^{-\frac{2}{n}|b-x|}< J^n_b(x)<J_b(x)\quad \forall x\in(-\infty, b);\\
		J_b(x)< J^{n}_b(x)<J_b(x)\cdot e^{\frac{1}{n}|b-x|}\quad \forall x\in(b,\infty).
	\end{cases}
	\ee 
	Then the first line in \eqref{eq.eg.fjb} together with the second line in \eqref{eq.eg1.0} indicates that, for all $n\in \N$, 
	\be\label{eq.eg.x>b} 
	J^n_b(x)\geq f(x)\quad \forall x\in [b,\infty).
	\ee 
Similarly, we can show that for any $n\in\N$,
	\begin{equation}\label{e7}
		J^n_{a,b}(x)\geq J_{a,b}(x)\cdot e^{-\frac{2}{n}(a-x)}\geq f(x),\quad \forall x\in (-\infty, a).
	\end{equation}
	As for $x\in(a,b)$, by the formulas of $J^n_{a,b}$ on $[a,b]$ in \eqref{eq.eg.jn} and $J_{a,b}$ on $[a,b]$ in \eqref{eq.eg.jab}, we have that
	$$J_{a,b}^{n}\to J_{a,b}\quad\text{and}\quad (J_{a,b}^{n})'\to J_{a,b}',\quad\text{uniformly on }(a,b).$$
	By the formula of $f$ on $[a,b]$ in \eqref{eq.eg.f}, we have $f'(a+)<J_{a,b}'(a+)$ and $f'(b-)>J_{a,b}'(b-)$. Then there exist $\alpha>0$ and $a',b'\in(a,b)$ with $a'<b'$, such that for any $x\in[a',b']$,
	$$f'(x)-J_{a,b}'(x)\leq-\alpha\ \ \text{for }x\in(a,a'); \ \ f'(x)-J_{a,b}'(x)\geq \alpha\ \text{for } x\in(b',b); \;\text{and } f(x)-J_{a,b}(x)\leq -\alpha.$$
	Then for $n\in\N$ large enough, we have that
	$$f'(x)\leq (J_{a,b}^n)'(x)\; \text{for } x\in(a,a');\ \ f'(x)\geq (J_{a,b}^n)'(x)\ \text{for }  x\in(b',b);\ \ f(x)\leq J_{a,b}^n(x)\ \text{for } x\in[a',b'],$$
	which implies that
	\begin{equation}\label{e8}
		J_{a,b}^n(x)\geq f(x)\quad \forall x\in(a,b),\quad \text{for $n\in \N$ large enough}. 
	\end{equation}
Then by \eqref{eq.eg.x>b}, \eqref{e7} and \eqref{e8}, $\{a,b\}$ is a mild equilibrium for $Q^n$ when $n\in\N$ is large enough.
	
		Note that $S^*(f, Q^n)\neq \{b\}$ for $n\in\N$. Indeed, if for some  $n_0\in \N$, $\{b\}$ is a mild equilibrium for $Q^{n_0}$, then due to the setup of $(\mu^n)_{n\in \overline \N}$, 
		$$f(a)\leq J^{Q^{n_0}}(a,\{b\},f)<J^{Q^\infty}(a,\{b\},f)=f(a),$$
		a contradiction. Then by the facts that $b\in S^*(f,Q^n)$ for all $n\in \overline \N$ and $\{a,b\}$ is a mild equilibrium for all $n\in \N$ large enough, we conclude that $S^*(f, Q^n)= \{a, b\}$ for $n\in \N$ large enough.


	Finally, since $S^*(f, Q^n)=\{a,b\}$ for all $n\in \N$ and $S^*(f, Q^\infty)=\{b\}$, we have that for $x<a$, 
	$$
	\begin{aligned}
		&V^{Q^\infty}(x, f)-\limsup_{n\to\infty} V^{Q^n}(x, f)=J_b(x)-\limsup_{n\to\infty}J^{Q^n}(x, \{a,b\},f)\\
		&=J_b(x)-\limsup_{n\to\infty}J^{Q^n}(x, \{a\},f)= J_b(x)-J^{Q^\infty}(x, \{a\},f)=J_b(x)-J_{a,b}(x)>0,
	\end{aligned}
	$$
	which implies \eqref{eq.eg.}.
\end{proof}

\subsection{An example showing $S^*(f^\infty, Q^\infty)\nsubseteq \liminf_{n\to\infty} S^*(f^n,Q^n)$}\label{eg2}

Let $\mu^n\equiv0$ and $\sigma^n\equiv 1$ for all $n\in \overline \N$. That is, we fix the process $X$ to be a one-dimensional Brownian motion and take $\X=\R$. Let $\delta(t)=\frac{1}{1+\beta t}$ with $\beta>0$, and set $\alpha:=1/\int_0^\infty e^{-s} \sqrt{2\beta s} ds$. We further define
$$
f^\infty(x):= 
\begin{cases}
	\begin{aligned}
		&x+\alpha, \quad & x\in[-\alpha,0),\\
		&-x+\alpha, \quad & x\in [0, \alpha),\\
		&0, \quad & \text{otherwise},
	\end{aligned}
\end{cases}\quad\text{and}\quad
f^n(x):= f^\infty \left(x-\frac{1}{n}\right)\; \forall\, n\in\N.
$$
Notice that the conditions in Theorem \ref{thm.continue.eps} are satisfied.
\begin{Proposition}
	\be\label{eq.eg2} 
	\limsup_{n\to\infty} S^*(f^n,Q^n)= \emptyset\subsetneqq \{0\}=S^*(f^\infty, Q^\infty).
	\ee
\end{Proposition}
\begin{proof}
	It is easy to see that $0\leq f^n\leq \alpha=\|f^n\|_\infty$ for any $n\in \overline \N$, which indicates that
	$$
	\frac{1}{n}\in S^*(f^n,Q^n)\quad \forall\, n\in \overline\N.
	$$
	By a calculation similar to \eqref{eq.eg.jblm} and the definition of $\alpha$, for any $n\in \overline\N$, we can get $J^{Q^n}(x, f^n)$ and its derivatives as follows:
	\be\label{eq.eg2.0} 
	\begin{cases}
		J^{Q^n}(x, \{\frac{1}{n}\}, f^n)= \alpha \int_0^\infty e^{-s}e^{-|x-\frac{1}{n}|\sqrt{2\beta s}}ds\quad \forall x\in \R;\\
		(J^{Q^n})'(\frac{1}{n}-, \{\frac{1}{n}\}, f^n)=1, \quad (J^{Q^n})'(\frac{1}{n}+, \{\frac{1}{n}\}, f^n)=-1;\\
		(J^{Q^n})''(x, \{\frac{1}{n}\}, f^n)=\alpha \int_0^\infty e^{-s}\cdot 2\beta s\cdot e^{-|x-\frac{1}{n}|\sqrt{2\beta s}}ds>0\quad \forall x\in (-\infty, \frac{1}{n})\cup (\frac{1}{n}, \infty).
	\end{cases}
	\ee 
	Notice that 
	\be\label{eq.eg2.1} 
	(f^n)'(x)=1, \; \text{for all}\; x\in \left(-\alpha+\frac{1}{n}, +\frac{1}{n}\right), \quad (f^n)'(x)=-1,  \; \text{for all}\;  x\in \left(\frac{1}{n}, \alpha+\frac{1}{n}\right).
	\ee 
	Then \eqref{eq.eg2.0} and \eqref{eq.eg2.1} together imply that
	$$
	J^{Q^n}\left(x, \left\{\frac{1}{n}\right\}, f^\infty\right)\geq f^n(x)\quad \forall n\in \overline\N,
	$$
	as shown in Figure \ref{figure2}.
	Therefore, $S^*(f^n,Q^n)=\{\frac{1}{n}\}$ for any $n\in \overline\N$, and \eqref{eq.eg2} follows.
\end{proof}

\begin{figure}
	\centering
		\includegraphics[width=0.9\linewidth]{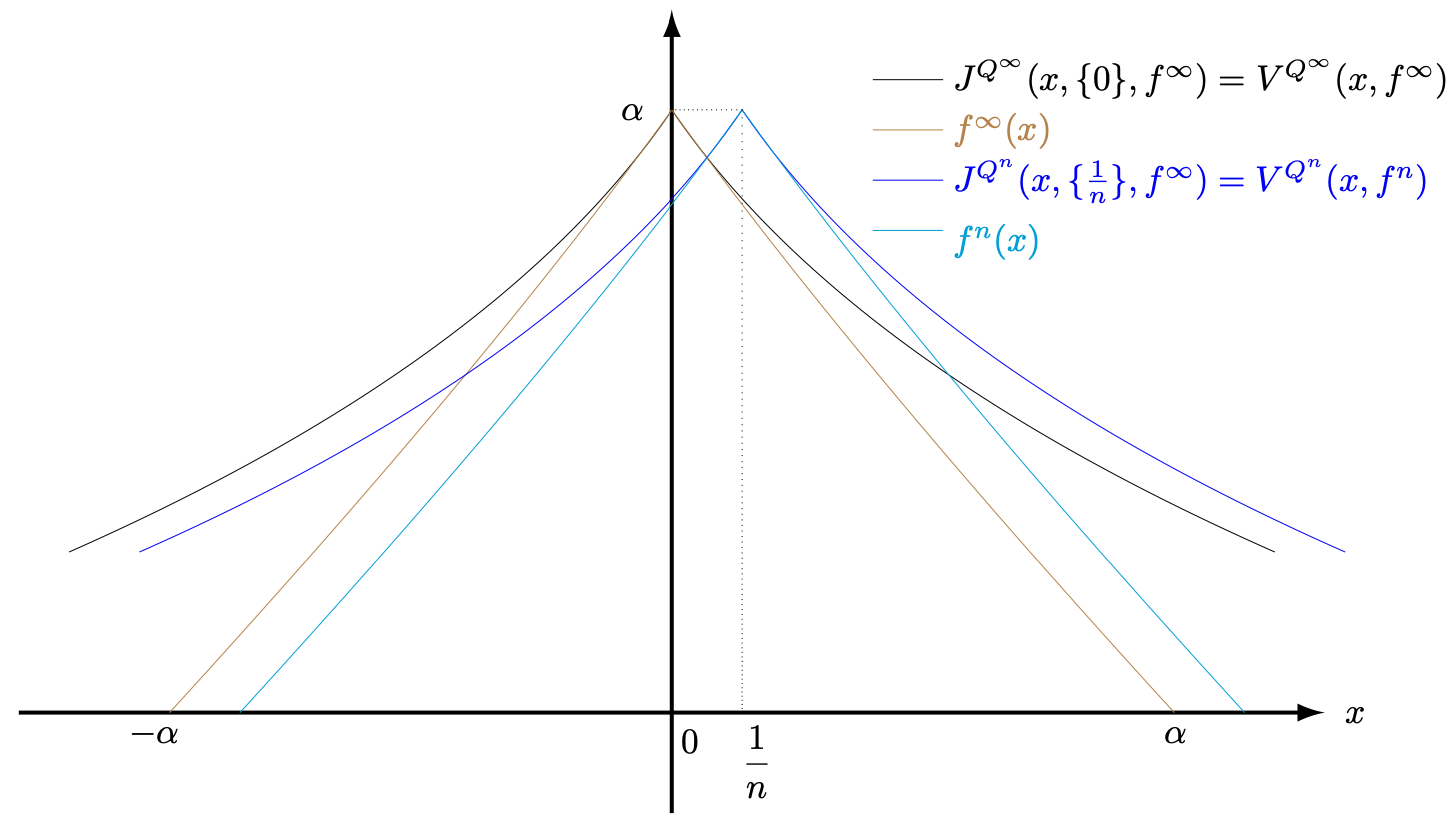}
		\caption{Relations of $V^{Q^n}(\cdot,f^n)$, $f^n(\cdot)$ $\forall n\in \overline\N$.}
		\label{figure2}
\end{figure}

\appendix
\section{A Brief Introduction of Mild, Weak and Strong Equilibria}\label{appendix}
Recall the dynamic of $X$ in \eqref{e1}. When $\delta$ is non-exponential (e.g., $\delta$ is a weighted discount function as in Assumption \ref{assume.delta.weight}), the optimal stopping problem
\be\label{eq.appen.stopping} 
	\sup\limits_{\tau} \E[\delta(\tau)f(X_{\tau})]
\ee 
can be time-inconsistent in the sense that an optimal stopping rule obtained today may no longer be optimal from a future's perspective. One approach to address this time inconsistency is to look for a subgame perfect Nash equilibrium, a strategy such that given future selves following this strategy, the current self has no incentive to deviate from it. For stopping problem \eqref{eq.appen.stopping}, there are mainly three different types of pure equilibria that are investigated in the literature, which we will introduce briefly as follows.

\begin{Definition}\label{def.mild}
	A closed set $S\subset\X$ is said to be  a mild equilibrium (w.r.t. $f$ and $Q$), if
	\begin{numcases}{}
		\label{e3'} f(x)\leq J^{Q}(x,S,f)\quad \forall x\notin S,	\\
		\label{e4}f(x)\geq J^{Q}(x,S,f)\quad \forall x\in S.
	\end{numcases}
\end{Definition}

This kind of equilibrium is first proposed in \cite{huang2018time}, and is called \textit{mild equilibrium} in \cite{MR4205889} to distinguish from other equilibrium concepts. Here $f(x)$ is the value for immediate stopping, while $J^Q(x,S,f)$ represents the value for continuing as $\rho_S$ is the first time to enter $S$ after time $0$. Hence, the economic meaning of condition \eqref{e3'} is clear: when $x\notin S$,  there is no incentive to switch from the action of ``continuing" to ``stopping" since the value $J$ is better than the value $f$. A similar reasoning seems also hold for the other case $x\in S$ in \eqref{e4}. However, {\color{black}given the current one-dimensional setting,} we have $\rho_S=0$ $\P^x$-a.s., and thus $f(x)=J^{Q}(x,S,f)$ for any $x\in S$. Therefore, \eqref{e4} holds trivially and we only require \eqref{e3'} for mild equilibrium as in the Definition \ref{def.mild.optimal}. On the other hand, in multi-dimensional setting, if $X_0$ belongs to the inner part of $S$, then $\rho_S=0$ a.s.; if $X_0$ is at the boundary of $S$, then the identity $\rho_S=0$ requires certain regularity of the boundary, and consequently, the verification of \eqref{e4} on the boundary may not be trivial.

Another type of pure equilibrium concept for time inconsistent stopping, called weak equilibrium, is proposed in \cite{MR3880244}, and is defined as follows. 
\begin{Definition}\label{def.weak}
	A closed set $S\subset\X$ is said to be  a weak equilibrium (w.r.t. $f$ and $Q$), if
	\begin{numcases}{}
		\label{e6} f(x)\leq J^Q(x,S,f)\quad \forall x\notin S,	\\
		\label{e7'}\underset{\varepsilon \searrow 0}{\liminf} \dfrac{f(x)-\E^Q_x[\delta(\rho^\varepsilon_S) f(X_{\rho^\varepsilon_S})]}{\varepsilon} \geq 0\quad \forall x\in S,
	\end{numcases}
	where
	$
	\rho^\varepsilon_S:=\inf\{t\geq \varepsilon: X_t\in S \}.
	$
\end{Definition}
 Compared to Definition \ref{def.mild}, the condition \eqref{e4} (which trivially holds for a one-dimensional recurrent process) is replaced with a first order condition \eqref{e6} for a weak equilibrium.

Recently, \cite{MR4205889} proposed another notion of equilibria as follows. 
\begin{Definition}\label{def.strong}
	A closed set $S\subset\X$ is said to be  a strong equilibrium, if
	\begin{numcases}{}
		\notag f(x)\leq J^Q(x,S,f)\quad \forall x\notin S,	\\
		\label{e8'}\exists\varepsilon(x)>0,\ \text{s.t.}\ \forall \varepsilon'\leq \varepsilon(x), f(x)-\E^Q_x[\delta(\rho^{\varepsilon'}_S) f(X_{\rho^{\varepsilon'}_S})] \geq 0\quad \forall x\in S.
	\end{numcases}
\end{Definition}
Compared to Definition \ref{def.weak}, the first order condition \eqref{e6} is upgraded to a local maximum condition \eqref{e8'}, which better captures the economic meaning of ``Nash equilibrium''. 

By the definitions above, a strong equilibrium must be weak, and a weak equilibrium must be mild. Under continuous time Markov chain setting, \cite{MR4205889} shows that an optimal mild equilibrium is also weak and strong. Recently, such relation has been further extended to the one-dimensional diffusion context in \cite{bayraktar2022equilibria}. The purpose of this paper is to consider the stability of the optimal value induced by all mild equilibria (which is also the value associated with the optimal mild equilibrium as indicated below Lemma \ref{lm.optimalmild}) w.r.t. the dynamic of process $X$ and reward function $f$ in a one-dimensional diffusion setting.

\bibliographystyle{plain}
\bibliography{reference}

\begin{thebibliography}{10}

\bibitem{bayraktar2022equilibria}
Erhan Bayraktar, Zhenhua Wang, and Zhou Zhou.
\newblock Equilibria of time-inconsistent stopping for one-dimensional
  diffusion processes.
\newblock {\em arXiv preprint arXiv:2201.07659}, 2022.

\bibitem{bayraktar2022stability}
Erhan Bayraktar, Zhenhua Wang, and Zhou Zhou.
\newblock Stability of equilibria in time-inconsistent stopping problems.
\newblock {\em arXiv preprint arXiv:2205.08656, to appear in in SIAM J. Control
  Optim.}, 2022.

\bibitem{MR3980261}
Erhan Bayraktar, Jingjie Zhang, and Zhou Zhou.
\newblock Time consistent stopping for the mean-standard deviation
  problem---the discrete time case.
\newblock {\em SIAM J. Financial Math.}, 10(3):667--697, 2019.

\bibitem{MR4205889}
Erhan Bayraktar, Jingjie Zhang, and Zhou Zhou.
\newblock Equilibrium concepts for time-inconsistent stopping problems in
  continuous time.
\newblock {\em Math. Finance}, 31(1):508--530, 2021.

\bibitem{bodnariu2022local}
Andi Bodnariu, S{\"o}ren Christensen, and Kristoffer Lindensj{\"o}.
\newblock Local time pushed mixed stopping and smooth fit for time-inconsistent
  stopping problems.
\newblock {\em arXiv preprint arXiv:2206.15124}, 2022.

\bibitem{MR1912205}
Andrei~N. Borodin and Paavo Salminen.
\newblock {\em Handbook of {B}rownian motion---facts and formulae}.
\newblock Probability and its Applications. Birkh\"{a}user Verlag, Basel,
  second edition, 2002.

\bibitem{MR3880244}
S\"{o}ren Christensen and Kristoffer Lindensj\"{o}.
\newblock On finding equilibrium stopping times for time-inconsistent
  {M}arkovian problems.
\newblock {\em SIAM J. Control Optim.}, 56(6):4228--4255, 2018.

\bibitem{MR4080735}
S\"{o}ren Christensen and Kristoffer Lindensj\"{o}.
\newblock On time-inconsistent stopping problems and mixed strategy stopping
  times.
\newblock {\em Stochastic Process. Appl.}, 130(5):2886--2917, 2020.

\bibitem{MR4276004}
S\"{o}ren Christensen and Kristoffer Lindensj\"{o}.
\newblock Time-inconsistent stopping, myopic adjustment and equilibrium
  stability: with a mean-variance application.
\newblock In {\em Stochastic modeling and control}, volume 122 of {\em Banach
  Center Publ.}, pages 53--76. Polish Acad. Sci. Inst. Math., Warsaw, 2020.

\bibitem{MR4124420}
Sebastian Ebert, Wei Wei, and Xun~Yu Zhou.
\newblock Weighted discounting---on group diversity, time-inconsistency, and
  consequences for investment.
\newblock {\em J. Econom. Theory}, 189:105089, 40, 2020.

\bibitem{feinstein2020continuity}
Zachary Feinstein.
\newblock Continuity and sensitivity analysis of parameterized {N}ash games.
\newblock {\em Econ. Theory Bull.}, 10(2):233--249, 2022.

\bibitem{feinstein2022dynamic}
Zachary Feinstein, Birgit Rudloff, and Jianfeng Zhang.
\newblock Dynamic set values for nonzero-sum games with multiple equilibriums.
\newblock {\em Mathematics of Operations Research}, 47(1):616--642, 2022.

\bibitem{huang2018time}
Yu-Jui Huang and Adrien Nguyen-Huu.
\newblock Time-consistent stopping under decreasing impatience.
\newblock {\em Finance Stoch.}, 22(1):69--95, 2018.

\bibitem{MR4250561}
Yu-Jui Huang and Zhenhua Wang.
\newblock Optimal equilibria for multidimensional time-inconsistent stopping
  problems.
\newblock {\em SIAM J. Control Optim.}, 59(2):1705--1729, 2021.

\bibitem{MR4273542}
Yu-Jui Huang and Xiang Yu.
\newblock Optimal stopping under model ambiguity: a time-consistent equilibrium
  approach.
\newblock {\em Math. Finance}, 31(3):979--1012, 2021.

\bibitem{MR3911711}
Yu-Jui Huang and Zhou Zhou.
\newblock The optimal equilibrium for time-inconsistent stopping problems---the
  discrete-time case.
\newblock {\em SIAM J. Control Optim.}, 57(1):590--609, 2019.

\bibitem{MR4116459}
Yu-Jui Huang and Zhou Zhou.
\newblock Optimal equilibria for time-inconsistent stopping problems in
  continuous time.
\newblock {\em Math. Finance}, 30(3):1103--1134, 2020.

\bibitem{MR2001996}
Bernt \O~ksendal.
\newblock {\em Stochastic differential equations}.
\newblock Universitext. Springer-Verlag, Berlin, sixth edition, 2003.
\newblock An introduction with applications.

\bibitem{MR4332966}
Ken~Seng Tan, Wei Wei, and Xun~Yu Zhou.
\newblock Failure of smooth pasting principle and nonexistence of equilibrium
  stopping rules under time-inconsistency.
\newblock {\em SIAM J. Control Optim.}, 59(6):4136--4154, 2021.

\end{thebibliography}

\end{document}